\documentclass{article}

\usepackage{amssymb, amsmath, amscd, amsthm, hyperref}

\usepackage{setspace}
\usepackage{geometry}
\usepackage[dvips,pdftex]{graphicx}
\graphicspath{{pictures/}}

\usepackage{tkz-graph}
\usepackage{scalefnt}
\usepackage{capt-of}

\author{I.~A.~Krepkiy}

\title{The sandpile groups of graphs of classes $CH_n (a_1, ..., a_n)$.}

\begin{document}

\vskip-3cm

\maketitle
\thispagestyle{empty}
\vskip-8cm
\rightline{}
\vskip8cm
\vskip-7cm{
\small{
\begin{center}
\vskip-1.5cm
Mathematics and Mechanics Faculty,\\
St. Petersburg State University \\
\end{center}

}
}
\begin{center}
{\em \href{mailto:feb418@gmail.com}{\texttt{feb418@gmail.com}}}
\end{center}

\vspace{3cm}

\begin{abstract}
\noindent The article considers the procedure of connection of graphs to the edges of a cyclic graph and its influence on the sandpile group of the graph thus obtained. A series of classes of graphs $CH_n(a_1,...,a_n)$ is defined. Recurrent and non-recurrent formulas for calculating the sandpile groups of all graphs of classes $CH_n (a_1, ..., a_n)$ are proposed.
\vskip 10pt
\noindent\textbf{Keywords:} sandpile groups, Smith normal form, classes of graphs $CH_n (a_1, ..., a_n)$.
\end{abstract}

\newpage

\section{Notations}

Let $M$ be an arbitrary square integer matrix. Denote the multiset of numbers arranged on the diagonal of the Smith normal form of $M$ (\cite{Smith}) by $\overline{M}$.\\
We will use (as in \cite{Krep}) the definition of the sandpile group of a graph in terms of the Smith normal form of the Laplacian matrix of this graph:
\newtheorem{definition}{Definition}
\begin{definition}
The sandpile group of a graph $G$ is the group $S(G) \cong \bigoplus \limits_{a \in (\overline{M} \setminus \{0\})} C_a$, where $M$ is a Laplacian matrix of $G$. (Here and below $C_n$ is a cyclic group of order $n$.)\\
\end{definition}
(More detailed information on the two classic definitions of the sandpile groups can be found in \cite{Def}.)\\
One can write $S(G) \cong \bigoplus \limits_{a \in \overline{M'}} C_a$, where $M'$ is obtained from $M$ by removing the row and column containing an arbitrary diagonal element (due to the properties of the Laplacian matrix).\\
Also, $\overline{A} = \overline{B}$, where the matrix $B$ is obtained from the matrix $A$ by a single operation of addition/subtraction of one column/row to another (due to the properties of the Smith normal form). Hence, the permutation of rows or columns of a matrix, as well as their multiplication by $-1$ does not affect the Smith normal form of this matrix.\\

\section{Main theorem}

Let $F$ and $G$ be some graphs. Let $f_1$, $f_2$ be non-negative integer-valued functions on the set of vertices of $F$, and $g_1$, $g_2$ be non-negative integer-valued functions on the set of vertices of $G$. Denote by $T$ a cyclic graph consisting of $n$ vertices. We can label the vertices of these graphs by natural numbers as follows:\\
1. $r$ vertices of $F$ are numbered from $1$ to $r$\\
2. $n$ vertices of $T$ are numbered from $r+1$ to $r+n$ in the order they appear in the loop\\
3. $s$ vertices of $G$ are numbered from $r+n+1$ to $r+n+s$\\

Now, for $i \in \mathbb{Z}, 0 \leq i \leq n-2$ we construct a new graph $H_i$ as follows:\\
1. connect each vertex $v$ of $G$ with $r+n-1$-th vertex of $T$ by $g_1(v)$ edges\\
2. connect each vertex $v$ of $G$ with $r+n$-th vertex of $T$ by $g_2(v)$ edges\\
3. connect each vertex $v$ of $F$ with $r+i$-th vertex of $T$ by $f_1(v)$ edges (if $1 \leq i \leq n-2$)\\
or connect each vertex $v$ of $F$ with $r+n$-th vertex of $T$ by $f_1(v)$ edges (if $i = 0$)\\
4. connect each vertex $v$ of $F$ with $r+i+1$-th vertex of $T$ by $f_2(v)$ edges\\

For example, if the graphs $F$, $T$, $G$ are as shown in Fig.$\ref{ris:PartF}-\ref{ris:PartG}$ and functions $f_1, f_2, g_1, g_2$ are such that\\
$f_1(1)=1, f_1(2)=0, f_1(3)=0$,\\
$f_2(1)=0, f_2(2)=0, f_2(3)=2$,\\
$g_1(10)=1, g_1(11)=1$,\\
$g_2(10)=1, g_2(11)=1$\\
then the resulting graphs $H_1$ and $H_2$ are shown in Fig. $\ref {ris:Sum1}, \ref {ris:Sum2}$.

\begin{center}
\begin{minipage}[h]{0.3\linewidth}
\center{
\begin{tikzpicture}[scale=0.6]
\GraphInit[vstyle=Classic]
\tikzset{VertexStyle/.style = {shape = circle,fill = black,minimum size = 6pt,inner sep=0pt}}
\Vertex[x=0,y=0]{1}
\Vertex[x=-1,y=0.5,Lpos=180]{2}
\Vertex[x=0,y=1]{3}
\Edges(1,2,3,1)
\end{tikzpicture}
}
\captionof{figure}{$F$}
\label{ris:PartF}
\end{minipage}
\hfill
\begin{minipage}[h]{0.3\linewidth}
\center{
\begin{tikzpicture}[scale=0.6]
\GraphInit[vstyle=Classic]
\tikzset{VertexStyle/.style = {shape = circle,fill = black,minimum size = 6pt,inner sep=0pt}}
\Vertex[x=0,y=0]{9}
\Vertex[x=-1,y=-0.5,Lpos=-90]{4}
\Vertex[x=-2,y=0,Lpos=180]{5}
\Vertex[x=-2,y=1,Lpos=180]{6}
\Vertex[x=-1,y=1.5,Lpos=90]{7}
\Vertex[x=0,y=1]{8}
\Edges(4,5,6,7,8,9,4)
\end{tikzpicture}
}
\captionof{figure}{$T$}
\label{ris:PartT}
\end{minipage}
\hfill
\begin{minipage}[h]{0.3\linewidth}
\center{
\begin{tikzpicture}[scale=0.6]
\GraphInit[vstyle=Classic]
\tikzset{VertexStyle/.style = {shape = circle,fill = black,minimum size = 6pt,inner sep=0pt}}
\Vertex[x=0,y=0,Lpos=180]{10}
\Vertex[x=1,y=0]{11}
\Edges(10,11)
\end{tikzpicture}
}
\captionof{figure}{$G$}
\label{ris:PartG}
\end{minipage}
\end{center}

\begin{center}
\begin{minipage}[h]{0.4\linewidth}
\center{
\begin{tikzpicture}[scale=0.6]
\GraphInit[vstyle=Classic]
\tikzset{VertexStyle/.style = {shape = circle,fill = black,minimum size = 6pt,inner sep=0pt}}
\Vertex[x=0,y=-1,Lpos=-90]{1}
\Vertex[x=-1,y=-0.5,Lpos=180]{2}
\Vertex[x=0,y=0,Lpos=90]{3}
\Edges(1,2,3,1)
\Vertex[x=3,y=0,Lpos=-90]{9}
\Vertex[x=2,y=-0.5,Lpos=-90]{4}
\Vertex[x=1,y=0,Lpos=30]{5}
\Vertex[x=1,y=1,Lpos=120]{6}
\Vertex[x=2,y=1.5,Lpos=90]{7}
\Vertex[x=3,y=1,Lpos=90]{8}
\Edges(4,5,6,7,8,9,4)
\Vertex[x=4,y=0.5,Lpos=180]{10}
\Vertex[x=5,y=0.5]{11}
\Edges(10,11)
\Edge[style={bend left}](3)(5) \Edge[style={bend right}](3)(5)
\Edge(1)(4)
\Edge(9)(10) \Edge(8)(10)
\Edge[style={bend right}](9)(11)  \Edge[style={bend left}](8)(11)
\end{tikzpicture}
}
\captionof{figure}{$H_1$}
\label{ris:Sum1}
\end{minipage}
\hfill
\begin{minipage}[h]{0.4\linewidth}
\center{
\begin{tikzpicture}[scale=0.6]
\GraphInit[vstyle=Classic]
\tikzset{VertexStyle/.style = {shape = circle,fill = black,minimum size = 6pt,inner sep=0pt}}
\Vertex[x=0,y=0,Lpos=-90]{1}
\Vertex[x=-1,y=0.5,Lpos=180]{2}
\Vertex[x=0,y=1,Lpos=90]{3}
\Edges(1,2,3,1)
\Vertex[x=3,y=0,Lpos=-90]{9}
\Vertex[x=2,y=-0.5,Lpos=-90]{4}
\Vertex[x=1,y=0,Lpos=-90]{5}
\Vertex[x=1,y=1,Lpos=90]{6}
\Vertex[x=2,y=1.5,Lpos=90]{7}
\Vertex[x=3,y=1,Lpos=90]{8}
\Edges(4,5,6,7,8,9,4)
\Vertex[x=4,y=0.5,Lpos=180]{10}
\Vertex[x=5,y=0.5]{11}
\Edges(10,11)
\Edge[style={bend left}](3)(6) \Edge[style={bend right}](3)(6)
\Edge(1)(5)
\Edge(9)(10) \Edge(8)(10)
\Edge[style={bend right}](9)(11)  \Edge[style={bend left}](8)(11)
\end{tikzpicture}
}
\captionof{figure}{$H_2$}
\label{ris:Sum2}
\end{minipage}
\end{center}

\renewcommand{\proofname}{Proof}

\newtheorem{thm}{Theorem}
\begin{thm}
The structure of the sandpile group of $H_i$ does not depend on the choice of $i$.
\end{thm}
\begin{proof}
It suffices to show that $S(H_k) \cong S(H_ {k +1})$ $\forall k \in \mathbb{Z}, 0 \leq k \leq n-3$.

Here we only consider the case $1 \leq k \leq n-4$. (For cases with $k = 0$ and $k = n-3$ the proof does not change, but the matrices are somewhat different from those described below.)\\

We denote the Laplacian matrices of $H_k$ and $H_ {k +1}$ by $A$ and $B$, respectively:\\

\begin{center}
\begin{minipage}[h]{0.4\linewidth}
\center{
\scalefont{1}
$A=\left(
\begin{array}{ccc}
A_{1,1} & A_{1,2} & A_{1,3} \\
A_{2,1} & A_{2,2} & A_{2,3} \\
A_{3,1} & A_{3,2} & A_{3,3} \\
\end{array}
\right)$
}
\end{minipage}
\hfill
\begin{minipage}[h]{0.4\linewidth}
\center{
\scalefont{1}
$B=\left(
\begin{array}{ccc}
B_{1,1} & B_{1,2} & B_{1,3} \\
B_{2,1} & B_{2,2} & B_{2,3} \\
B_{3,1} & B_{3,2} & B_{3,3} \\
\end{array}
\right)$
}
\end{minipage}
\end{center}

Both matrices are divided into $9$ blocks.\\

Block $(1,1)$ describes the Laplacian matrix of $F$, which is modified by subtraction of $p_i$ from the diagonal elements.\\

\begin{figure}[h]
\center{
\scalefont{1}
$A_{1,1}=B_{1,1}=\left(
\begin{array}{ccccc}
-p_1 & & & & \\
 &-p_2 & & & \\
 & & \ddots & &\\
 & & &-p_{r-1} &\\
 & & & &-p_r\\
\end{array}
\right)$
}
\end{figure}

Block $(3,3)$ describes the Laplacian matrix of $G$, which is modified by subtraction of $q_i$ from the diagonal elements.\\

\begin{figure}[h]
\center{
\scalefont{1}
$A_{3,3}=B_{3,3}=\left(
\begin{array}{ccccc}
-q_1 & & & & \\
 &-q_2 & & & \\
 & & \ddots & &\\
 & & &-q_{s-1} &\\
 & & & &-q_s\\
\end{array}
\right)$
}
\end{figure}

(In blocks (1,1) and (3,3) we only show the numbers that should be subtracted from the diagonal elements of the corresponding matrices.)\\

Block $(2,2)$ describes the Laplacian matrix of $T$, which is modified by subtraction of $w', x', y', z'$ from the diagonal elements corresponding to the vertices connected to the vertices of $F$ and $G$.\\

\begin{figure}[h]
\center{
\scalefont{1}
$A_{2,2}=\left(
\begin{array}{cccccccccccc}
-2 & 1 & & & & & & & & & & 1 \\
1 & -2 & 1 & & & & & & & & & \\
 & 1 & \ddots & \ddots & & & & & & & & \\
 & & \ddots & -2 & 1 & & & & & & & \\
 & & & 1 & -2-w' & 1 & & & & & & \\
 & & & & 1 & -2-x' & 1 & & & & & \\
 & & & & & 1 & -2 & 1 & & & & \\
 & & & & & & 1 & -2 & 1 & & & \\
 & & & & & & & 1 & \ddots & \ddots & & \\
 & & & & & & & & \ddots & -2 & 1 & \\
 & & & & & & & & & 1 & -2-y' & 1 \\
 1 & & & & & & & & & & 1 & -2-z' \\
\end{array}
\right)$
}
\end{figure}

\begin{figure}[h]
\center{
\scalefont{1}
$B_{2,2}=\left(
\begin{array}{cccccccccccc}
-2 & 1 & & & & & & & & & & 1 \\
1 & -2 & 1 & & & & & & & & & \\
 & 1 & \ddots & \ddots & & & & & & & & \\
 & & \ddots & -2 & 1 & & & & & & & \\
 & & & 1 & -2 & 1 & & & & & & \\
 & & & & 1 & -2-w' & 1 & & & & & \\
 & & & & & 1 & -2-x' & 1 & & & & \\
 & & & & & & 1 & -2 & 1 & & & \\
 & & & & & & & 1 & \ddots & \ddots & & \\
 & & & & & & & & \ddots & -2 & 1 & \\
 & & & & & & & & & 1 & -2-y' & 1 \\
 1 & & & & & & & & & & 1 & -2-z' \\
\end{array}
\right)$
}
\end{figure}

Here the numbers $x', y', z', w'$ are equal to the number of edges that connect the graphs $F$ and $G$ with the respective vertices of the cycle $T$.\\

Blocks $(1,2)$ and $(2,1)$ contain numbers $x_i$ and $w_i$, which correspond to the number of edges connecting different vertices of $F$ with two vertices of $T$.\\

Block $(1,2)$:

\begin{figure}[h]
\center{
\scalefont{1}
$\left(
\begin{array}{cccccccc}
0 & \cdots & 0 & w_1 & x_1 & 0 & \cdots & 0 \\
\vdots & \ddots &  \vdots& \vdots & \vdots & \vdots & \ddots & \vdots \\
0 & \cdots & 0 & w_r & x_r & 0 & \cdots & 0 \\
\end{array}
\right)$
}
\end{figure}

The nonzero columns of $A_{1,2}$ are numbered as $k$ and $k+1$, and the nonzero columns of $B_{1,2}$ are numbered as $k+1$ and $k+2$. $A_{2,1}=A^\intercal_{1,2}$, $B_{2,1}=B^\intercal_{1,2}$.\\

\newpage

Blocks $(2,3)$ and $(3,2)$ contains numbers $y_i$ and $z_i$, which correspond to the number of edges connecting different vertices of $G$ with two vertices of $T$.\\

\begin{figure}[h]
\center{
\scalefont{1}
$A_{3,2}=B_{3,2}=A^\intercal_{2,3}=B^\intercal_{2,3}=\left(
\begin{array}{ccccc}
0 & \cdots & 0 & y_1 & z_1 \\
\vdots& \ddots & \vdots & \vdots & \vdots \\
0 & \cdots & 0 & y_s & z_s \\
\end{array}
\right)$
}
\end{figure}

Blocks $(1,3)$ and $(3,1)$ are empty.\\

\begin{figure}[h]
\center{
\scalefont{1}
$A_{1,3}=B_{1,3}=A^\intercal_{3,1}=B^\intercal_{3,1}=\left(
\begin{array}{ccc}
0 & \cdots & 0 \\
\vdots& \ddots & \vdots \\
0 & \cdots & 0 \\
\end{array}
\right)$
}
\end{figure}

Here $w' = \sum \limits_{1 \leq i \leq r} w_i$, $x' = \sum \limits_{1 \leq i \leq r} x_i$, $y' = \sum \limits_{1 \leq i \leq s} y_i$, $z' = \sum \limits_{1 \leq i \leq s} z_i$, $p_i=x_i+w_i$, $q_i=y_i+z_i$.\\

The main difference between these matrices (matrices $A$ and $B$) is the difference between numbers of rows and columns that contain numbers $w_i, x_i, y_i, z_i$.\\

Let us denote by $A'$ the matrix obtained from $A$ by removing the $(n+r)$-th row and column:\\
\begin{figure}[h]
\center{
\scalefont{1}
$A'=\left(
\begin{array}{ccc}
A'_{1,1} & A'_{1,2} & A'_{1,3} \\
A'_{2,1} & A'_{2,2} & A'_{2,3} \\
A'_{3,1} & A'_{3,2} & A'_{3,3} \\
\end{array}
\right)$
}
\end{figure}

Here $A'_{1,1}=A_{1,1}$, $A'_{1,3}=A_{1,3}$, $A'_{3,1}=A_{3,1}$, $A'_{3,3}=A_{3,3}$.\\

Block $(1,2)$:

\begin{figure}[!h]
\center{
\scalefont{1}
$A'_{1,2} = A'^\intercal_{2,1} =\left(
\begin{array}{cccccccc}
0 & \cdots & 0 & w_1 & x_1 & 0 & \cdots & 0 \\
\vdots & \ddots & \vdots& \vdots & \vdots & \vdots & \ddots & \vdots \\
0 & \cdots & 0 & w_r & x_r & 0 & \cdots & 0 \\
\end{array}
\right)$
}
\end{figure}

The nonzero columns of $A'_{1,2}$ are numbered as $k$ and $k+1$.\\
\newpage

Block $(3,2)$:

\begin{figure}[h]
\center{
\scalefont{1}
$A'_(3,2)=A'^\intercal_{2,3}=\left(
\begin{array}{ccccc}
0 & \cdots & 0 & y_1 \\
\vdots& \ddots & \vdots & \vdots \\
0 & \cdots & 0 & y_s \\
\end{array}
\right)$
}
\end{figure}

Block $(2,2)$:

\begin{figure}[h]
\center{
\scalefont{1}
$A'_{2,2}=\left(
\begin{array}{ccccccccccc}
-2 & 1 & & & & & & & & & \\
1 & -2 & 1 & & & & & & & & \\
 & 1 & \ddots & \ddots & & & & & & & \\
 & & \ddots & -2 & 1 & & & & & & \\
 & & & 1 & -2-w' & 1 & & & & & \\
 & & & & 1 & -2-x' & 1 & & & & \\
 & & & & & 1 & -2 & 1 & & & \\
 & & & & & & 1 & -2 & 1 & & \\
 & & & & & & & 1 & \ddots & \ddots & \\
 & & & & & & & & \ddots & -2 & 1 \\
 & & & & & & & & & 1 & -2-y' \\
\end{array}
\right)$
}
\end{figure}

Let us denote by $B'$ the matrix obtained from $B$ by removing the $(r+1)$-th row and column:\\

\begin{figure}[h]
\center{
\scalefont{1}
$B'=\left(
\begin{array}{ccc}
B'_{1,1} & B'_{1,2} & B'_{1,3} \\
B'_{2,1} & B'_{2,2} & B'_{2,3} \\
B'_{3,1} & B'_{3,2} & B'_{3,3} \\
\end{array}
\right)$
}
\end{figure}

Here $B'_{1,1}=B_{1,1}$, $B'_{1,3}=B_{1,3}$, $B'_{3,1}=B_{3,1}$, $B'_{3,3}=B_{3,3}$.\\

Blocks $(1,2)$ and $(2,1)$ coincide with the respective blocks of the matrix $A'$:\\
$B'_{1,2}=A'_{1,2}$, $B'_{2,1}=B'^\intercal_{1,2}$.\\

Block $(3,2)$:

\begin{figure}[!h]
\center{
\scalefont{1}
$B'_{3,2}=B'^\intercal_{2,3}=\left(
\begin{array}{ccccc}
0 & \cdots & 0 & y_1 & z_1 \\
\vdots& \ddots & \vdots & \vdots & \vdots \\
0 & \cdots & 0 & y_s & z_s \\
\end{array}
\right)$
}
\end{figure}

\newpage

Block $(2,2)$:

\begin{figure}[!h]
\center{
\scalefont{1}
$B'_{2,2}=\left(
\begin{array}{ccccccccccc}
-2 & 1 & & & & & & & & & \\
1 & \ddots & \ddots & & & & & & & & \\
 & \ddots & -2 & 1 & & & & & & & \\
 & & 1 & -2 & 1 & & & & & & \\
 & & & 1 & -2-w' & 1 & & & & & \\
 & & & & 1 & -2-x' & 1 & & & & \\
 & & & & & 1 & -2 & 1 & & & \\
 & & & & & & 1 & \ddots & \ddots & & \\
 & & & & & & & \ddots & -2 & 1 & \\
 & & & & & & & & 1 & -2-y' & 1 \\
 & & & & & & & & & 1 & -2-z' \\
\end{array}
\right)$
}
\end{figure}

We need to show that $\overline{A'} = \overline{B'}$. To do this, we perform a series of actions with columns and rows of $B'$.\\

We add the last $s+1$ rows to the $(r+n-2)$-th row. Only the blocks $(2,2)$ and $(2,3)$ get changed:

\begin{figure}[h]
\center{
\scalefont{1}
$(2,2)=\left(
\begin{array}{ccccccccccc}
-2 & 1 & & & & & & & & & \\
1 & \ddots & \ddots & & & & & & & & \\
 & \ddots & -2 & 1 & & & & & & & \\
 & & 1 & -2 & 1 & & & & & & \\
 & & & 1 & -2-w' & 1 & & & & & \\
 & & & & 1 & -2-x' & 1 & & & & \\
 & & & & & 1 & -2 & 1 & & & \\
 & & & & & & 1 & \ddots & \ddots & & \\
 & & & & & & & \ddots & -2 & 1 & \\
 & & & & & & & & 1 & -1 & -1 \\
 & & & & & & & & & 1 & -2-z' \\
\end{array}
\right)$
}
\end{figure}

\begin{figure}[h]
\center{
\scalefont{1}
$(2,3)=\left(
\begin{array}{ccc}
0 & \cdots & 0\\
\vdots & \ddots & \vdots \\
0 & \cdots & 0\\
z_1 & \cdots & z_s \\
\end{array}
\right)$
}
\end{figure}

\newpage

We add the last $s+1$ columns to the $(r+n-2)$-th column. Only the blocks $(2,2)$ and $(3,2)$ get changed:

\begin{figure}[h]
\center{
\scalefont{1}
$(2,2)=\left(
\begin{array}{ccccccccccc}
 -2 & 1 & & & & & & & & & \\
 1 & \ddots & \ddots & & & & & & & & \\
 & \ddots & -2 & 1 & & & & & & & \\
 & & 1 & -2 & 1 & & & & & & \\
 & & & 1 & -2-w' & 1 & & & & & \\
 & & & & 1 & -2-x' & 1 & & & & \\
 & & & & & 1 & -2 & 1 & & & \\
 & & & & & & 1 & \ddots & \ddots & & \\
 & & & & & & & \ddots & -2 & 1 & \\
 & & & & & & & & 1 & -2 & -1 \\
 & & & & & & & & & -1 & -2-z' \\
\end{array}
\right)$
}
\end{figure}

\begin{figure}[h]
\center{
\scalefont{1}
$(3,2)=\left(
\begin{array}{cccc}
0 & \cdots & 0 & z_1 \\
\vdots & \ddots & \vdots & \vdots \\
0 & \cdots & 0 & z_s \\
\end{array}
\right)$
}
\end{figure}

We add the last $s$ rows to the $(r+n-1)$-th row. Only the blocks $(2,2)$ and $(2,3)$ get changed:

\begin{figure}[!h]
\center{
\scalefont{1}
$(2,2)=\left(
\begin{array}{ccccccccccc}
-2 & 1 & & & & & & & & & \\
1 & \ddots & \ddots & & & & & & & & \\
 & \ddots & -2 & 1 & & & & & & & \\
 & & 1 & -2 & 1 & & & & & & \\
 & & & 1 & -2-w' & 1 & & & & & \\
 & & & & 1 & -2-x' & 1 & & & & \\
 & & & & & 1 & -2 & 1 & & & \\
 & & & & & & 1 & \ddots & \ddots & & \\
 & & & & & & & \ddots & -2 & 1 & \\
 & & & & & & & & 1 & -2 & -1 \\
 & & & & & & & & & -1 & -2 \\
\end{array}
\right)$
}
\end{figure}

\begin{figure}[!h]
\center{
\scalefont{1}
$(2,3)=\left(
\begin{array}{ccc}
0 & \cdots & 0\\
\vdots & \ddots & \vdots\\
0 & \cdots & 0\\
 -y_1 & \cdots & -y_s \\
\end{array}
\right)$
}
\end{figure}

\newpage

We add the last $s$ columns to the $(r+n-1)$-th column. Only the blocks $(2,2)$ and $(3,2)$ get changed:

\begin{figure}[h]
\center{
\scalefont{1}
$(2,2)=\left(
\begin{array}{ccccccccccc}
-2 & 1 & & & & & & & & & \\
1 & \ddots & \ddots & & & & & & & & \\
 & \ddots & -2 & 1 & & & & & & & \\
 & & 1 & -2 & 1 & & & & & & \\
 & & & 1 & -2-w' & 1 & & & & & \\
 & & & & 1 & -2-x' & 1 & & & & \\
 & & & & & 1 & -2 & 1 & & & \\
 & & & & & & 1 & \ddots & \ddots & & \\
 & & & & & & & \ddots & -2 & 1 & \\
 & & & & & & & & 1 & -2 & -1 \\
 & & & & & & & & & -1 & -2-y' \\
\end{array}
\right)$
}
\end{figure}

\begin{figure}[h]
\center{
\scalefont{1}
$(3,2)=\left(
\begin{array}{cccc}
0 & \cdots & 0 & -y_1 \\
\vdots & \ddots & \vdots & \vdots \\
0 & \cdots & 0 & -y_s \\
\end{array}
\right)$
}
\end{figure}

Finally, we change signs of $(r+n-1)$-th row and column to obtain the matrix $A'$. Only the blocks $(2,2)$. $(2,3)$ and $(3,2)$ get changed:\\

\begin{figure}[!h]
\center{
\scalefont{1}
$(2,2)=\left(
\begin{array}{ccccccccccc}
-2 & 1 & & & & & & & & & \\
1 & \ddots & \ddots & & & & & & & & \\
 & \ddots & -2 & 1 & & & & & & & \\
 & & 1 & -2 & 1 & & & & & & \\
 & & & 1 & -2-w' & 1 & & & & & \\
 & & & & 1 & -2-x' & 1 & & & & \\
 & & & & & 1 & -2 & 1 & & & \\
 & & & & & & 1 & \ddots & \ddots & & \\
 & & & & & & & \ddots & -2 & 1 & \\
 & & & & & & & & 1 & -2 & 1 \\
 & & & & & & & & & 1 & -2-y' \\
\end{array}
\right)=A'_{2,2}$
}
\end{figure}

\begin{figure}[!h]
\center{
\scalefont{1}
$(2,3)=\left(
\begin{array}{ccc}
0 & \cdots & 0\\
\vdots & \ddots & \vdots\\
0 & \cdots & 0\\
y_1 & \cdots & y_s \\
\end{array}
\right)=A'_{2,3}$
}
\end{figure}

\begin{figure}[!h]
\center{
\scalefont{1}
$(3,2)=\left(
\begin{array}{cccc}
0 & \cdots & 0 & y_1 \\
\vdots & \ddots & \vdots & \vdots \\
0 & \cdots & 0 & y_s \\
\end{array}
\right)=A'_{3,2}$
}
\end{figure}

The blocks $(1,1)$, $(1,2)$, $(2,1)$, $(1,3)$, $(3,1)$, $(3,3)$ have not changed and coincide with the corresponding blocks of the matrix $A$.
So $\overline{A'} = \overline{B'}$ and the theorem is proved.

\end{proof}

There is also a weaker version of this theorem that is applicable to a smaller class of graphs. The proof of the second theorem avoids a cumbersome manipulations with matrices.\\

Let $F$ be an arbitrary planar graph and $f$ is one of its edges that is adjacent to the outer face of $F$. Let $G$ be an arbitrary planar graph and $g$ is one of its edges that is adjacent to the outer face of $G$. Let $T$ be a cyclic planar graph on $n$ vertices. The edges of $T$ are marked as $t_0, t_1, ..., t_ {n-1}$ in order they appear in the loop.\\

For $i \in \mathbb {Z}, 1 \leq i \leq n-1$ we construct a planar graph $J_i$ as follows:\\
1. We make a merge of edges $t_0$ and $g$ in such a way that all vertices of $G$ were on the outer face of $T$.
1. We make a merge of edges $t_i$ and $f$ in such a way that all vertices of $F$ were on the outer face of $T$.

We assume that the structure of $F$ and $G$ (as planar graphs) is preserved and the cycle of $T$ now has exactly $n-2$ edges that adjacent to the outer face of $J_i$.

\begin{thm}
The structure of the sandpile group of $J_i$ does not depends on the choice of $i$.
\end{thm}

\begin{proof}
It suffices to show that for every $2 \leq i \leq n-1$ the sandpile groups of $J_1$ and $J_i$ are isomorphic. There is a theorem which states that the sandpile groups of a planar graph and it's dual graph are isomorphic (\cite{Dual}).
It is obvious that the graphs $J'_1$ and $J'_i$ (that are duals, respectively, of $J_1$ and $J_i$) are isomorphic, which implies that $S(J_1) \cong S(J'_1) \cong S(J'_i) \cong S(J_i)$.
\end{proof}

\section{The graphs of classes $CH_n(a_1,...,a_n)$}

The theorem proved above can be used to calculate the sandpile groups for some series of graphs.
We recursively define a series of classes of graphs $CH_n(a_1, a_2, ..., a_n)$ (where $a_i \in \mathbb{Z}, a_i \geq 2, i \in [1..n]$). For convenience, we assume that every graph $G$ of each class is provided with an ordered subset of its own vertices: $L(G) = [v_1, ..., v_n]$.
$CH_1(a_1)$ contains only the cycle of $a_1$ vertices $v_1, ..., v_ {a_1}$, equipped with the natural order of the vertices in the cycle.\\
Every $G \in CH_{n +1}(a_1, a_2, ..., a_{n+1})$ is constructed from some graph $H \in CH_n(a_1, a_2, ..., a_n)$. Let $L(H)=[w_1, w_2, ..., w_k]$. We fix an arbitrary integer $i$, such that $1 \leq i <k$. Add to the graph $H$ a linear chain of $a_{n +1}-2$ vertices $v_1, ..., v_{a_{n+1}-2}$, connected by edges in accordance with the order of indices. Connect vertices $v_1$ and $w_i$ by one edge and connect vertices $v_{a_{n+1}-2}$ and $w_{i+1}$ by another edge. Now let $L(G)=[w_i, v_1, ..., v_{a_{n+1}-2}, w_{i+1}]$. (If $a_{n+1}=2$, we just need to add one more edge between vertices $w_i$ and $w_{i +1}$ and suppose that $L(G)=[w_i, w_{i+1}]$.) At this point the construction of $G$ is complete. For example, Fig. \ref{ris:constructExample} shows a graph of the class $CH_4(3,6,4,6)$ (here straight lines denote the edges of the subgraph of class $CH_3(3,6,4)$).\\

\begin{center}
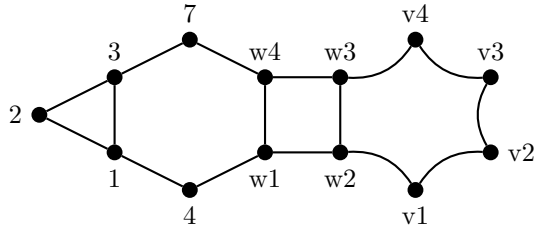

\center{
\begin{tikzpicture}[scale=1]
\GraphInit[vstyle=Classic]
\tikzset{VertexStyle/.style = {shape = circle,fill = black,minimum size = 6pt,inner sep=0pt}}
\Vertex[x=0,y=0,Lpos=-90]{1}
\Vertex[x=-1,y=0.5,Lpos=180]{2}
\Vertex[x=0,y=1,Lpos=90]{3}
\Edges(1,2,3,1)
\Vertex[x=1,y=-0.5,Lpos=-90]{4}
\Vertex[x=2,y=0,Lpos=-90]{w1}
\Vertex[x=2,y=1,Lpos=90]{w4}
\Vertex[x=1,y=1.5,Lpos=90]{7}
\Edges(1,4,w1,w4,7,3)
\Vertex[x=3,y=0,Lpos=-90]{w2}
\Vertex[x=3,y=1,Lpos=90]{w3}
\Edges(w1,w2,w3,w4)
\Vertex[x=4,y=-0.5,Lpos=-90]{v1}
\Vertex[x=5,y=0]{v2}
\Vertex[x=5,y=1,Lpos=90]{v3}
\Vertex[x=4,y=1.5,Lpos=90]{v4}
\Edges[style={bend left}](w2,v1,v2,v3,v4,w3)
\end{tikzpicture}
}
\captionof{figure}{$G \in CH_4(3,6,4,6)$}
\label{ris:constructExample}
\end{center}

Less formally, any graph $G \in CH_n(a_1, a_2, ..., a_n)$ consists of a ``chain'', obtained through a series of connection (by edges) of an ordered set of cyclic graphs with lengths $a_1, a_2, ..., a_n$. For example Fig. \ref{ris:exampleSimilar1} - \ref{ris:exampleSimilar3} show three different graphs belonging to the class $CH_4(3,6,4,6)$.\\
It is obvious (by main theorem) that for any $G, H \in CH_n(a_1, a_2, ..., a_n)$ we have $S(G) \cong S(H)$. For example, the sandpile group of each of three graphs shown in Fig. \ref{ris:exampleSimilar1} - \ref{ris:exampleSimilar3} has the structure of $C_{373}$.\\

\begin{center}
\begin{minipage}[h]{0.3\linewidth}
\center{
\begin{tikzpicture}[scale=0.6]
\GraphInit[vstyle=Classic]
\tikzset{VertexStyle/.style = {shape = circle,fill = black,minimum size = 6pt,inner sep=0pt}}
\Vertex[x=0,y=0,NoLabel]{1}
\Vertex[x=-1,y=0.5,NoLabel]{2}
\Vertex[x=0,y=1,NoLabel]{3}
\Edges(1,2,3,1)
\Vertex[x=1,y=-0.5,NoLabel]{4}
\Vertex[x=2,y=0,NoLabel]{5}
\Vertex[x=2,y=1,NoLabel]{6}
\Vertex[x=1,y=1.5,NoLabel]{7}
\Edges(1,4,5,6,7,3)
\Vertex[x=3,y=0,NoLabel]{8}
\Vertex[x=3,y=1,NoLabel]{9}
\Edges(5,8,9,6)
\Vertex[x=4,y=-0.5,NoLabel]{10}
\Vertex[x=5,y=0,NoLabel]{11}
\Vertex[x=5,y=1,NoLabel]{12}
\Vertex[x=4,y=1.5,NoLabel]{13}
\Edges(8,10,11,12,13,9)
\end{tikzpicture}
}
\captionof{figure}{$G \in CH_4(3,6,4,6)$}
\label{ris:exampleSimilar1}
\end{minipage}
\hfill
\begin{minipage}[h]{0.3\linewidth}
\center{
\begin{tikzpicture}[scale=0.6]
\GraphInit[vstyle=Classic]
\tikzset{VertexStyle/.style = {shape = circle,fill = black,minimum size = 6pt,inner sep=0pt}}
\Vertex[x=0,y=0,NoLabel]{1}
\Vertex[x=-1,y=0.5,NoLabel]{2}
\Vertex[x=0,y=1,NoLabel]{3}
\Edges(1,2,3,1)
\Vertex[x=1,y=-0.5,NoLabel]{4}
\Vertex[x=2,y=0,NoLabel]{5}
\Vertex[x=2,y=1,NoLabel]{6}
\Vertex[x=1,y=1.5,NoLabel]{7}
\Edges(1,4,5,6,7,3)
\Vertex[x=2,y=-1,NoLabel]{8}
\Vertex[x=3,y=-0.5,NoLabel]{9}
\Edges(4,8,9,5)
\Vertex[x=2,y=-2,NoLabel]{10}
\Vertex[x=3,y=-2.5,NoLabel]{11}
\Vertex[x=4,y=-2,NoLabel]{12}
\Vertex[x=4,y=-1,NoLabel]{13}
\Edges(8,10,11,12,13,9)
\end{tikzpicture}
}
\captionof{figure}{$G \in CH_4(3,6,4,6)$}
\label{ris:exampleSimilar2}
\end{minipage}
\hfill
\begin{minipage}[h]{0.3\linewidth}
\center{
\begin{tikzpicture}[scale=0.6]
\GraphInit[vstyle=Classic]
\tikzset{VertexStyle/.style = {shape = circle,fill = black,minimum size = 6pt,inner sep=0pt}}
\Vertex[x=0,y=0,NoLabel]{1}
\Vertex[x=-1,y=0.5,NoLabel]{2}
\Vertex[x=0,y=1,NoLabel]{3}
\Edges(1,2,3,1)
\Vertex[x=1,y=-0.5,NoLabel]{4}
\Vertex[x=2,y=0,NoLabel]{5}
\Vertex[x=2,y=1,NoLabel]{6}
\Vertex[x=1,y=1.5,NoLabel]{7}
\Edges(1,4,5,6,7,3)
\Vertex[x=2,y=-1,NoLabel]{8}
\Vertex[x=3,y=-0.5,NoLabel]{9}
\Edges(4,8,9,5)
\Vertex[x=0,y=-1,NoLabel]{10}
\Vertex[x=0,y=-2,NoLabel]{11}
\Vertex[x=1,y=-2.5,NoLabel]{12}
\Vertex[x=2,y=-2,NoLabel]{13}
\Edges(4,10,11,12,13,8)
\end{tikzpicture}
}
\captionof{figure}{$G \in CH_4(3,6,4,6)$}
\label{ris:exampleSimilar3}
\end{minipage}
\end{center}

For each class $CH_n(a_1, a_2, ..., a_n)$ we choose a canonical representative of this class --- a graph that is arranged in such a way that all of its $n$ main cycles have a common vertex. We will denote this graph by $Ch_n(a_1, a_2, ..., a_n)$. For example, the canonical representative of the class $CH_4(3,6,4,6)$ is shown in Fig. \ref{ris:exampleNum}.\\

\begin{center}
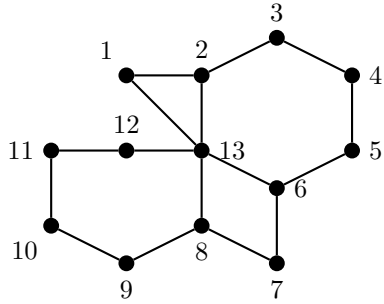

\center{
\begin{tikzpicture}[scale=1]
\GraphInit[vstyle=Classic]
\tikzset{VertexStyle/.style = {shape = circle,fill = black,minimum size = 6pt,inner sep=0pt}}
\Vertex[x=0,y=0,Lpos=120]{1}
\Vertex[x=1,y=0,Lpos=90]{2}
\Vertex[x=2,y=0.5,Lpos=90]{3}
\Vertex[x=3,y=0]{4}
\Vertex[x=3,y=-1]{5}
\Vertex[x=2,y=-1.5]{6}
\Vertex[x=2,y=-2.5,Lpos=-90]{7}
\Vertex[x=1,y=-2,Lpos=-90]{8}
\Vertex[x=0,y=-2.5,Lpos=-90]{9}
\Vertex[x=-1,y=-2,Lpos=-120]{10}
\Vertex[x=-1,y=-1,Lpos=180]{11}
\Vertex[x=0,y=-1,Lpos=90]{12}
\Edges(1,2,3,4,5,6,7,8,9,10,11,12)
\Vertex[x=1,y=-1]{13}
\Edge(1)(13) \Edge(2)(13) \Edge(6)(13) \Edge(8)(13) \Edge(12)(13)
\end{tikzpicture}
}
\captionof{figure}{$Ch_4(3,6,4,6)$}
\label{ris:exampleNum}
\end{center}

Now our task is to calculate the sandpile group of each of these canonical representatives.\\
What is the Laplacian matrix of this graph? Let $Ch_n(a_1, a_2, ..., a_n)$ consist of $k+1$ vertices. Let us enumerate them as follows. The only common vertex of all cycles is assigned the number $k+1$. All other vertices are given the numbers from $1$ to $k$ to match the order they appear in the ``outer'' cycle of the whole graph. An example of numbering of vertices of $Ch_4(3,6,4,6)$ is shown in Fig. \ref{ris:exampleNum}.\\

Obviously, in such a numbering the Laplacian matrix of the graph $Ch_n(a_1, a_2, ..., a_n)$ looks as follows:\\

\begin{figure}[h]
\center{
\scalefont{1}
$\left(
\begin{array}{c|c|c|c|c|c}
h_1 & 1 & & & & z_1\\
 \hline
1 & h_2 & 1 & & & z_2\\
 \hline
 & 1 & \ddots & \ddots & & \vdots \\
  \hline
 & & \ddots & h_{k-1} & 1 & z_{k-1}\\
  \hline
 & & & 1 & h_k & z_k\\
  \hline
z_1 & z_2 & \cdots & z_{k-1} & z_k & h_{k+1} \\
\end{array}
\right)$
}
\end{figure}

We remove the last column and the last row of the matrix and denote the resulting matrix by $M$:\\

\begin{figure}[h]
\center{
\scalefont{1}
$\left(
\begin{array}{c|c|c|c|c}
h_1 & 1 & & & \\
 \hline
1 & h_2 & 1 & & \\
 \hline
 & 1 & \ddots & \ddots & \\
  \hline
 & & \ddots & h_{k-1} & 1 \\
  \hline
 & & & 1 & h_k \\
\end{array}
\right)$
}
\end{figure}

To calculate the sandpile group of a graph $Ch_n(a_1, a_2, ..., a_n)$ it is sufficient to compute the Smith normal form of the matrix $M$.\\

We transform $M$ as follows:\\
1. From the second row we substract the first row $h_2$ times. The second row takes the form $(r_2, 0,1,0, ..., 0)$, where, $r_2 = 1-r_1 h_2$, $r_1 = h_1$.\\
2. For $i$ from $3$ to $k$ we repeat the following procedure:\\
From the the $i$-th row we substract the $(i_1)$-th row $h_i$ times, then once from the $i$-th row we subtract the $(i-2)$-th row.\\
As a result, the $i$-th row becomes $(r_i, 0, ..., 0,1,0, ..., 0)$, where the number $1$ is in the $(i+1)$-th position. $r_i=-r_{i-1} h_i-r_{i-2}$.\\
3. Finally we replace the last line to the first position.\\

As a result, the matrix takes the form:\\

\begin{figure}[h]
\center{
\scalefont{1}
$\left(
\begin{array}{c|c|c|c|c}
r_k & & & & \\
 \hline
r_1 & 1 & & & \\
 \hline
\vdots & & \ddots & & \\
  \hline
r_{k-2} & & & 1 & \\
  \hline
r_{k-1} & & & & 1 \\
\end{array}
\right)$
}
\end{figure}

It is clear that the ``unwanted'' elements of the first column can be removed by manipulation with other columns.
Therefore, thr Smith normal form of the matrix $M$ looks like:\\

\begin{figure}[h]
\center{
\scalefont{1}
$\left(
\begin{array}{c|c|c|c|c}
1 & & & & \\
 \hline
 & 1 & & & \\
 \hline
 & & \ddots & & \\
  \hline
 & & & 1 & \\
  \hline
 & & & & r_k \\
\end{array}
\right)$
}
\end{figure}

So the group of the graph $Ch_n(a_1, a_2, ..., a_n)$ is cyclic and we need only to determine its cardinality.\\
We use the well-known statement about the cardinality of the sandpile group of a graph. Namely, the cardinality of the sandpile group of a graph equals the number of spanning trees of this graph (\cite{Def}).\\

Here we count the number of spanning trees of a graph of class $CH_{n +1}(a_1, a_2, ..., a_n, a_{n+1})$.\\
We define a series of functions $F_i(x_1, x_2, ..., x_i), i \in \mathbb{N}$:\\
$$F_n(a_1,...,a_n)=|Ch_n(a_1,...,a_n)|$$\\

Next, recall that our graph is obtained from the graph of class $CH_n(a_1, a_2, ..., a_n)$ by adding of $a_{n+1}-2$ vertices and $a_{n +1}-1$ edges. These additional edges together with edge $u$ (which previously had belonged to the graph of class $CH_n(a_1, a_2, ..., a_n)$) constitute a cycle of length $a_{n+1}.$ (For example, Fig.\ref{ris:uCh} shows a graph of class $CH_4(3,6,4,6)$ with its edge $u$.)\\
To obtain a spanning tree, we need to remove some edges belonging to this cycle. Here we have two options:\\

1. Remove exactly one of $a_{n+1}-1$ edges (which differ from $u$) in the cycle of length $a_n$, (edge that did not belong to the graph of class $CH_n(a_1, a_2, ..., a_n)$). (Example of the result of such an operation on the graph of class $CH_4 (3,6,4,6)$ is shown in Fig.\ref{ris:uChMod1}.) It is clear that we can not remove more edges from this set, if we are going to get a spanning tree. We get a graph with number of spanning trees equal to the number of spanning trees of $Ch_n(a_1, a_2, ..., a_n)$, which means that here we have $(a_{n +1}-1) \cdot F_n(a_1 , a_2, ..., a_n)$ opportunities to construct a spanning tree.\\

\begin{center}
\begin{minipage}[h]{0.4\linewidth}
\center{
\begin{tikzpicture}[scale=1]
\GraphInit[vstyle=Classic]
\tikzset{VertexStyle/.style = {shape = circle,fill = black,minimum size = 6pt,inner sep=0pt}}
\Vertex[x=0,y=0,NoLabel]{1}
\Vertex[x=-1,y=0.5,NoLabel]{2}
\Vertex[x=0,y=1,NoLabel]{3}
\Edges(1,2,3,1)
\Vertex[x=1,y=-0.5,NoLabel]{4}
\Vertex[x=2,y=0,NoLabel]{w1}
\Vertex[x=2,y=1,NoLabel]{w4}
\Vertex[x=1,y=1.5,NoLabel]{7}
\Edges(1,4,w1,w4,7,3)
\Vertex[x=3,y=0,NoLabel]{w2}
\Vertex[x=3,y=1,NoLabel]{w3}
\Edges(w1,w2)
\Edges(w3,w4)
\Vertex[x=4,y=-0.5,NoLabel]{v1}
\Vertex[x=5,y=0,NoLabel]{v2}
\Vertex[x=5,y=1,NoLabel]{v3}
\Vertex[x=4,y=1.5,NoLabel]{v4}
\Edges(w2,v1,v2,v3,v4,w3)
\Edge[label=u](w2)(w3)
\end{tikzpicture}
}
\captionof{figure}{The graph of class $CH_4(3,6,4,6)$}
\label{ris:uCh}
\end{minipage}
\hfill
\begin{minipage}[h]{0.4\linewidth}
\center{
\begin{tikzpicture}[scale=1]
\GraphInit[vstyle=Classic]
\tikzset{VertexStyle/.style = {shape = circle,fill = black,minimum size = 6pt,inner sep=0pt}}
\Vertex[x=0,y=0,NoLabel]{1}
\Vertex[x=-1,y=0.5,NoLabel]{2}
\Vertex[x=0,y=1,NoLabel]{3}
\Edges(1,2,3,1)
\Vertex[x=1,y=-0.5,NoLabel]{4}
\Vertex[x=2,y=0,NoLabel]{w1}
\Vertex[x=2,y=1,NoLabel]{w4}
\Vertex[x=1,y=1.5,NoLabel]{7}
\Edges(1,4,w1,w4,7,3)
\Vertex[x=3,y=0,NoLabel]{w2}
\Vertex[x=3,y=1,NoLabel]{w3}
\Edges(w1,w2)
\Edges(w3,w4)
\Vertex[x=4,y=-0.5,NoLabel]{v1}
\Vertex[x=5,y=0,NoLabel]{v2}
\Vertex[x=5,y=1,NoLabel]{v3}
\Vertex[x=4,y=1.5,NoLabel]{v4}
\Edges(w2,v1,v2,v3)
\Edges(v4,w3)
\Edge[label=u](w2)(w3)
\end{tikzpicture}
}
\captionof{figure}{$CH_4(3,6,4,6) \rightarrow CH_3(3,6,4)$}
\label{ris:uChMod1}
\end{minipage}
\end{center}

2. Do not delete any of $a_{n+1}-1$ edges (which differ from $u$) of cycle of length $n$, (edges that not belonged to the graph of class $CH_n(a_1, a_2,..., a_n)$). In this case, the only way to get rid of the cycle is to remove the edge $u$. After removing it, we get a graph of class $CH_n(a_1, a_2, ..., a_{n-1}, a_n+a_{n+1}-2)$. (Fig.\ref{ris:uChMod2} shows a graph obtained from the graph of class $CH_4(3,6,4,6)$ by removing the edge $u$. The edges, which we agreed not to remove, are marked by curved lines. Fig.\ref{ris:uChMod3} shows the graph, obtained from the graph of class $CH_3 (3,6,8)$ by ``contraction'' of fixed edges.) But since we have agreed not to touch the $a_{n+1}-1$ edges, then it is clear that the number of spanning trees that we can get is the same as the number of spanning trees of the graph $Ch_n(a_1, a_2, ..., a_{n-1}, a_n-1)$, which means that here we have $F_n(a_1, a_2, ..., a_n-1)$ opportunities to construct a spanning tree. Generally speaking, there is an inaccuracy. There is no description of how to act in the case $a_n = 2$. This inaccuracy will be removed later.\\

\begin{center}
\begin{minipage}[h]{0.4\linewidth}
\center{
\begin{tikzpicture}[scale=1]
\GraphInit[vstyle=Classic]
\tikzset{VertexStyle/.style = {shape = circle,fill = black,minimum size = 6pt,inner sep=0pt}}
\Vertex[x=0,y=0,NoLabel]{1}
\Vertex[x=-1,y=0.5,NoLabel]{2}
\Vertex[x=0,y=1,NoLabel]{3}
\Edges(1,2,3,1)
\Vertex[x=1,y=-0.5,NoLabel]{4}
\Vertex[x=2,y=0,NoLabel]{w1}
\Vertex[x=2,y=1,NoLabel]{w4}
\Vertex[x=1,y=1.5,NoLabel]{7}
\Edges(1,4,w1,w4,7,3)
\Vertex[x=3,y=-0.5,NoLabel]{w2}
\Vertex[x=3,y=1.5,NoLabel]{w3}
\Edges(w1,w2)
\Edges(w3,w4)
\Vertex[x=4,y=-0.5,NoLabel]{v1}
\Vertex[x=5,y=0,NoLabel]{v2}
\Vertex[x=5,y=1,NoLabel]{v3}
\Vertex[x=4,y=1.5,NoLabel]{v4}
\Edges[style={bend left}](w2,v1,v2,v3,v4,w3)
\end{tikzpicture}
}
\captionof{figure}{A graph of class $CH_3(3,6,8)$}
\label{ris:uChMod2}
\end{minipage}
\hfill
\begin{minipage}[h]{0.4\linewidth}
\center{
\begin{tikzpicture}[scale=1]
\GraphInit[vstyle=Classic]
\tikzset{VertexStyle/.style = {shape = circle,fill = black,minimum size = 6pt,inner sep=0pt}}
\Vertex[x=0,y=0,NoLabel]{1}
\Vertex[x=-1,y=0.5,NoLabel]{2}
\Vertex[x=0,y=1,NoLabel]{3}
\Edges(1,2,3,1)
\Vertex[x=1,y=-0.5,NoLabel]{4}
\Vertex[x=2,y=0,NoLabel]{w1}
\Vertex[x=2,y=1,NoLabel]{w4}
\Vertex[x=1,y=1.5,NoLabel]{7}
\Edges(1,4,w1,w4,7,3)
\Vertex[x=3,y=0.5,NoLabel]{x1}
\Edges(w1,x1,w4)
\end{tikzpicture}
}
\captionof{figure}{A graph of class $CH_3(3,6,3)$}
\label{ris:uChMod3}
\end{minipage}
\end{center}

It is also clear that the cyclic graph $Ch_1(a_1)$ has exactly $a_1$ spanning trees. So now the functions $F_i$ are defined as follows:\\

\begin{equation} \label{eqF}
F_1(x_1)=x_1, F_{i+1}(x_1,x_2,...,x_i,x_{i+1})=F_i(x_1,x_2,...,x_i-1)+(x_{i+1}-1) \cdot F_i(x_1,x_2,...,x_i)
\end{equation}

Generally speaking, such a definition of the functions $F_n$ is ``redundant'' in the sense that their values are determined by including the cases where some of the arguments are equal to $1$ (despite the fact that in the description of $Ch_n(a_1,. .., a_n)$ all of the arguments are larger a $1$). This feature is used to eliminate the inaccuracies that arose in counting the spanning trees of a graph. Suppose we have a graph of the class $CH_n(a_1, a_2, ..., a_k, 2, 2, ..., 2, 2, m) $, where $a_k> 2$. We can calculate the corresponding value of $F_n$, using the new recursive definition:\\
$F_n(a_1,a_2,...,a_k,2,2,...,2,2,m)=(m-1) \cdot F_{n-1}(a_1,a_2,...,a_k,2,2,...,2,2)+F_{n-1}(a_1,a_2,...,a_k,2,2,...,2,1)$ =
$(m-1) \cdot F_{n-1}(a_1,a_2,...,a_k,2,2,...,2,2)+(1-1) \cdot F_{n-2}(a_1,a_2,...,a_k,2,2,...,2)+F_{n-2}(a_1,a_2,...,a_k,2,2,...,1)$=...=
$(m-1) \cdot F_{n-1}(a_1,a_2,...,a_k,2,2,...,2,2)+F_{k}(a_1,a_2,...,a_k-1)$.
This result is consistent with the process of counting the spanning trees:\\
1. The number $(m-1) \cdot F_{n-1}(a_1, a_2, ..., a_k, 2,2, ..., 2,2)$ corresponds to the case when we remove one of $m-1$ edges (which differ from $u$) in a cycle of length $m$.\\
2. When we intend to save the $m-1$ edges, we are obliged to remove not only the $u$, but all other edges connecting vertices that are connected by the edge $u$. When we do that, we get a graph of class $CH_{k}(a_1, a_2, ..., a_k+m-2)$. Since we have agreed not to touch $m-1$ of its edges, then the number of spanning trees that we can get, equals the number of spanning trees of a graph $Ch_{k}(a_1, a_2, ..., a_k-1)$. This number corresponds precisely to $F_ {k}(a_1, a_2, ..., a_k-1)$.\\

Hence the functions $F_n$ are well-defined.\\

Now we find the non-recursive representation of functions $F_n$. To do this, we need some auxiliary objects.\\

For non-negative integers $i, j$ define the set $C_{i, j}$.\\
$I \in C_{i, j}$ if and only if all of the following conditions are true:\\
1. $I \subset \mathbb{N}$.\\
2. $|I|=i$.\\
3. $x \in I \Rightarrow 1 \leq x \leq j$.\\
4. $(a,b \in I, a < b \ | \ (\forall t \in (a,b) \Rightarrow t \not\in I)) \Rightarrow (b-a \equiv 1 \ (mod 2))$.\\
5. $j-max(I) \equiv 0 (mod 2)$.\\

We also define the sets $A_ {i, j}$ and $B_ {i, j}$ as follows:\\
1. $C_{i,j} = A_{i,j} \bigsqcup B_{i,j}$.\\
2. $I \in B_{i,j} \Leftrightarrow j \in I$.\\

We define the functions $\alpha_{i, j}, \beta_{i, j}, \gamma_{i, j}$ and $\beta_{i, j}'$:\\
$\alpha_{i,j}(x_1,...,x_j) = \sum \limits_{I \in A_{i,j}} 1*\prod \limits_{i \in I} x_i$.\\
$\beta_{i,j}(x_1,...,x_j) = \sum \limits_{I \in B_{i,j}} 1*\prod \limits_{i \in I} x_i$.\\
$\gamma_{i,j}(x_1,...,x_j) = \sum \limits_{I \in C_{i,j}} 1*\prod \limits_{i \in I} x_i$.\\
$\beta_{i,j}'(x_1,...,x_{j-1}) = \beta_{i,j}(x_1,...,x_{j-1},1)$.\\

Thus, for example, with $i = 3$, $j = 7$ these functions are:\\
$\alpha_{3,7}(x_1,...,x_7)=x_1 x_2 x_3 + x_1 x_2 x_5 + x_1 x_4 x_5 + x_3 x_4 x_5$\\
$\beta_{3,7}(x_1,...,x_7)=x_1 x_2 x_7 + x_1 x_4 x_7 + x_1 x_6 x_7 + x_3 x_4 x_7 + x_3 x_6 x_7 + x_5 x_6 x_7$\\
$\gamma_{3,7}(x_1,...,x_7)=x_1 x_2 x_3 + x_1 x_2 x_5 + x_1 x_2 x_7 + x_1 x_4 x_5 + x_1 x_4 x_7 + x_1 x_6 x_7 + x_3 x_4 x_5 + x_3 x_4 x_7 + x_3 x_6 x_7 + x_5 x_6 x_7$\\
$\beta_{3,7}'(x_1,...,x_6) = x_1 x_2 + x_1 x_4 + x_1 x_6 + x_3 x_4 + x_3 x_6 + x_5 x_6$\\

It is clear that $\gamma_{i, j} = \alpha_{i, j} + \beta_{i, j} $.\\
It is also easy to check that for any positive $i, j$ the following equation is true:\\
\begin{equation} \label{gammaEq}
\gamma_{i,j}=x_j \cdot \gamma_{i-1,j-1}+\beta_{i+1,j-1}'
\end{equation}

For $n \in \mathbb{N} $ we define a function $G_n$:\\

\begin{equation} \label{eqG}
G_n(x_1, ..., x_n) = \gamma_{n, n} - \gamma_{n-2, n} + \gamma_{n-4, n} - \gamma_{n-6, n} + ... \; ,
\end{equation}

where the sum extends while the first subscript of functions $\gamma$ preserves the non-negativity.\\
To be more specific:\\
1. $G_{4k}=\gamma_{4k,4k}-\gamma_{4k-2,4k}+...+\gamma_{0,4k}$\\
2. $G_{4k+1}=\gamma_{4k+1,4k+1}-\gamma_{4k-1,4k+1}+...+\gamma_{1,4k+1}$\\
3. $G_{4k+2}=\gamma_{4k+2,4k+2}-\gamma_{4k,4k+2}+...-\gamma_{0,4k+2}$\\
4. $G_{4k+3}=\gamma_{4k+3,4k+3}-\gamma_{4k+1,4k+3}+...-\gamma_{1,4k+3}$\\

\begin{thm}
$G_n(x_1,...,x_n) = F_n(x_1,...,x_n)$
\end{thm}
\begin{proof}
$G_1(x_1)=\gamma_{1,1}=x_1=F_1(x_1)$.\\
Suppose that $G_i = F_i$. We must show that  $G_{i+1} = F_{i+1}$.
To do this, according to (\ref{eqF}), it suffices to show that $G_{i+1}(x_1,x_2,...,x_i,x_{i+1})=F_i(x_1,x_2,...,x_i-1)+(x_{i+1}-1) \cdot F_i(x_1,x_2,...,x_i)=$[by the induction hypothesis]$=G_i(x_1,x_2,...,x_i-1)+(x_{i+1}-1) \cdot G_i(x_1,x_2,...,x_i)$.\\
Indeed, consider, for example, the case $i = 4k$ (for the remaining cases the chain of equalities is built in the same way):\\
$G_i(x_1,x_2,...,x_i-1)+(x_{i+1}-1) \cdot G_i(x_1,x_2,...,x_i)=G_{4k}(x_1,x_2,...,x_{4k}-1)+(x_{4k+1}-1) \cdot G_{4k}(x_1,x_2,...,x_{4k})=(\gamma_{4k,4k}-\gamma_{4k-2,4k}+...+\gamma_{0,4k})(x_1,x_2,...,x_{4k}-1)+(x_{4k+1} \cdot \gamma_{4k,4k}-x_{4k+1} \cdot \gamma_{4k-2,4k}+...-x_{4k+1} \cdot \gamma_{2,4k}+x_{4k+1} \cdot \gamma_{0,4k})-(\gamma_{4k,4k}-\gamma_{4k-2,4k}+...-\gamma_{2,4k}+\gamma_{0,4k})=(\alpha_{4k,4k}-
\alpha_{4k-2,4k}+...+\alpha_{0,4k})(x_1,x_2,...,x_{4k}-1)+
(\beta_{4k,4k}-\beta_{4k-2,4k}+...+\beta_{0,4k})(x_1,x_2,...,x_{4k}-1)+
(x_{4k+1} \cdot \gamma_{4k,4k}-x_{4k+1} \cdot \gamma_{4k-2,4k}+...-x_{4k+1} \cdot \gamma_{2,4k}+x_{4k+1} \cdot \gamma_{0,4k})-(\gamma_{4k,4k}-\gamma_{4k-2,4k}+...-\gamma_{2,4k}+\gamma_{0,4k})=
(\alpha_{4k,4k}-\alpha_{4k-2,4k}+...+\alpha_{0,4k})+
(\beta_{4k,4k}-\beta_{4k-2,4k}+...+\beta_{0,4k})-
(\beta_{4k,4k}'-\beta_{4k-2,4k}'+...+\beta_{0,4k}')+
(x_{4k+1} \cdot \gamma_{4k,4k}-x_{4k+1} \cdot \gamma_{4k-2,4k}+...-x_{4k+1} \cdot \gamma_{2,4k}+x_{4k+1} \cdot \gamma_{0,4k})-(\gamma_{4k,4k}-\gamma_{4k-2,4k}+...-\gamma_{2,4k}+\gamma_{0,4k})=
(\gamma_{4k,4k}-\gamma_{4k-2,4k}+...-\gamma_{2,4k}+\gamma_{0,4k})+
(\beta_{4k,4k}'-\beta_{4k-2,4k}'+...+\beta_{0,4k}')+(x_{4k+1} \cdot \gamma_{4k,4k}-x_{4k+1} \cdot \gamma_{4k-2,4k}+...-x_{4k+1} \cdot \gamma_{2,4k}+x_{4k+1} \cdot \gamma_{0,4k})-(\gamma_{4k,4k}-\gamma_{4k-2,4k}+...-\gamma_{2,4k}+\gamma_{0,4k})=
(\beta_{4k,4k}'-\beta_{4k-2,4k}'+...+\beta_{0,4k}')+
(x_{4k+1} \cdot \gamma_{4k,4k}-x_{4k+1} \cdot \gamma_{4k-2,4k}+...-x_{4k+1} \cdot \gamma_{2,4k}+x_{4k+1} \cdot \gamma_{0,4k})=x_{4k+1} \cdot \gamma_{4k,4k}-(\beta_{4k,4k}'+x_{4k+1} \cdot \gamma_{4k-2,4k})+...+(\beta_{2,4k}'+x_{4k+1} \cdot \gamma_{0,4k})-\beta_{0,4k}'=
[by (\ref{gammaEq})]=
x_1 x_2 ... x_{4k} x_{4k+1}-
\gamma_{4k-1,4k+1}+...
+\gamma_{1,4k+1}-
0=\gamma_{4k+1,4k+1}-\gamma_{4k+1,4k+1}+...+\gamma_{1,4k+1}=G_{4k+1}=G_{i+1}$
\end{proof}

Thus, we can formulate the last theorem.

\begin{thm}
$S(Ch_n(a_1,...,a_n)) \cong C_{F_n(a_1,...,a_n)} \cong C_{G_n(a_1,...,a_n)})$, where $F_n$, $G_n$ are defined by (\ref{eqF}), (\ref{eqG}).
\end{thm}

\newpage

\section{Acknowledgements}
The author wishes first to thank his
supervisor, Dr. Sergei V. Duzhin for valuable advices and
guidance. Secondly, the author is grateful to Vladimir O. Zolotov for some interesting conversations related to the subject of this article.
The work on this paper was supported by RFBR, grant 13-01-00383.


\end{document}